\documentclass[12pt,reqno]{amsart}
\usepackage[utf8]{inputenc}
\usepackage[T1]{fontenc}% Pour pouvoir taper les accent directement et non pas passer par
\usepackage[usenames, dvipsnames]{color}
\usepackage{ulem}

\usepackage{dsfont, amsfonts, amsmath, amssymb,amscd, stmaryrd, latexsym, amsthm, dsfont}
\usepackage[frenchb,english]{babel}
\usepackage{enumerate}
\usepackage{longtable}
\usepackage{geometry}
\usepackage{float}
\usepackage{tikz}
\usetikzlibrary{shapes,arrows}
\usepackage{float}
\usepackage{tikz}
\usepackage{xypic}
\usetikzlibrary{shapes,arrows}

\newtheorem{theorem}{Theorem}[section]
\newtheorem{lemma}[theorem]{Lemma}
\newtheorem{proposition}[theorem]{Proposition}
\newtheorem{corollary}[theorem]{Corollary}

\theoremstyle{remark}
\newtheorem{remark}[theorem]{\bf Remark}

\usepackage{hyperref}
\hypersetup{
	colorlinks=true,
	urlcolor=blue,
	citecolor=blue}
\def\NN{\mathds{N}}
\def\RR{\mathbb{R}}
\def\QQ{\mathbb{Q}}
\def\CC{\mathbb{C}}
\def\ZZ{\mathbb{Z}}
\def\kk{\mathds{k}}
\begin{document}
	
\def\NN{\mathbb{N}}
\def\RR{\mathds{R}}
\def\HH{I\!\! H}
\def\QQ{\mathbb{Q}}
\def\CC{\mathds{C}}
\def\ZZ{\mathbb{Z}}
\def\DD{\mathds{D}}
\def\OO{\mathcal{O}}
\def\kk{\mathds{k}}
\def\KK{\mathbb{K}}
\def\ho{\mathcal{H}_0^{\frac{h(d)}{2}}}
\def\LL{\mathbb{L}}
\def\L{\mathds{k}_2^{(2)}}
\def\M{\mathds{k}_2^{(1)}}
\def\k{\mathds{k}^{(*)}}
\def\l{\mathds{L}}

\selectlanguage{english}
\title[On the Hilbert $2$-class field towers...]{On the Hilbert $2$-class field towers of some cyclotomic $\mathbb{Z}_2$-extensions}
%premier auteur

 \author[M. M. Chems-Eddin]{Mohamed Mahmoud Chems-Eddin}
 \address{Mohamed Mahmoud CHEMS-EDDIN: Mohammed First University, Mathematics Department, Sciences Faculty, Oujda, Morocco }
 \email{2m.chemseddin@gmail.com}

 % troisième auteur
 \author[A. Zekhnini]{Abdelkader Zekhnini}
 \address{Abdelkader Zekhnini: Mohammed First University, Mathematics Department, Pluridisciplinary faculty, Nador, Morocco}
 \email{zekha1@yahoo.fr}

 \author[A. Azizi]{Abdelmalek Azizi}
 \address{Abdelmalek Azizi: Mohammed First University, Mathematics Department, Sciences Faculty, Oujda, Morocco }
 \email{abdelmalekazizi@yahoo.fr}


\begin{abstract}
In this paper, we study the length of the $2$-class field towers and the structure of the Galois groups $\mathrm{Gal}(\mathcal{L}(K_n)/K_n)$ of the maximal unramified $2$-extensions   of the layers $K_n$ of the cyclotomic $\ZZ_2$-extension of some special Dirichlet fields. The capitulation problem is investigated too.
\end{abstract}

\selectlanguage{english}

\maketitle

 \section{\textbf{Introduction}}
 \label{Sec:1}

  Let $K$ be a number field and $\mathrm{Cl}_2(K)$ its $2$-class group, that is the $2$-Sylow subgroup of its class group in the wide sense.
  Let  $K^{(1)}$ denote    the Hilbert $2$-class field of $K$, that is the maximal unramified (including the infinite primes) abelian extension of $K$ whose degree is a power of $2$. For any positive integer $i\geq1$, define inductively $K^{(i)} $ as $K^{(0)}= K$ and
    $K^{(i)}=(K^{(i-1)})^{(1)}$. Then the sequence of fields
  $$K=K^{(0)} \subset K^{(1)}\subset  K^{(2)}  \subset \cdots\subset K^{(i)}\subset \cdots $$
  is called   the $2$-class field tower of $K$.
  If for all $i\geq1$,   $K^{(i)}\neq K^{(i-1)}$,  the tower is said to be infinite,  otherwise the tower is said to be  finite,  and the minimal integer $i$ such that
    $K^{(i)}= K^{(i-1)}$ is called the length of the tower.
  The determination of the length of the $2$-class field tower of $K$ and the structure of the Galois groups of the tower  are until nowadays two classical and difficult open
    problems   of class field theory.
     Actually, there is no known decision procedure to determine whether or not the $2$-class field
    of an arbitrary number field $K$ is finite.  However, it is known that if the rank of $\mathrm{Cl}_2(K^{(1)})\leq 2$, then by group theoretic  means,
    the tower is finite and its length is at most $3$ (cf. \cite[Theorem 1]{Bl}). If the $2$-class group of $K$ is of type $(2,2)$ (here $(a_1,...,a_r)$ denotes the direct sum of cyclic groups of order $a_i$, for $i=1,...,r$), then the Hilbert class field tower of such field terminates in at most two steps and the Galois groups of the tower are characterized by means of capitulation problem in the three unramified quadratic extension of $K$ (cf. \cite{kisilvsky}). If  the $2$-class group of $K$ is of type $(2,2^n)$, with $n\geq 2$, the situation becomes more difficult, and  many authors treated it using the capitulation problem  in the unramified extensions of $K$ for fields of  low degree (e.g. $2$ or $4$, cf. \cite{taous2008, benlem99,benlem06, benlemsne, bensney}). Our contribution in this paper is to determine the length of the $2$-class field tower, the structure of its Galois group and to study the capitulation problem for some families of number fields of large degree over $\mathbb{Q}$ whose $2$-class groups are  of type $(2,2^n)$, with $n\geq 2$.
    The novelty in our  procedure to deal with this problem is the combination of some tools from  Iwasawa theory and the rank of some unramified extensions.

   Let  $d $ be  an odd positive square-free integer. Consider the  biquadratic number field $K=\mathbb{Q}({\sqrt{d}, i })$.  	Denote by 	$K_\infty$ the cyclotomic    $\mathbb Z_2$-extension of $K$,  that is  a Galois extension of $K$ whose Galois group is topologically isomorphic to the $2$-adic ring $\mathbb{Z}_2$. It is well known that for any integer $n\geq0$, the field $K_\infty$ contains a unique cyclic subfield $K_n$  of degree $2^n$ over $K$  called the $n$-th	layer of the $\mathbb Z_2$-extension of $K$. Denote by $A(K_n)$ (or simply $A_n$) the $2$-part of the ideal class group of  $K_n$ and let $\mathcal{L}(K_n)$ be the maximal unramified $2$-extension of $K_{n}$.
   The cyclotomic $\ZZ_2$-extension of $K$ is given by the sequence $K_n=K_{n, d}=\mathbb{Q}({\zeta_{2^{n+2}}, \sqrt{d}})$ such that
 $$  K_{1, d}=\mathbb{Q}({\zeta_{2^{3}}, \sqrt{d}  }) \subset   K_{2, d}=\mathbb{Q}({ \zeta_{2^{4}}, \sqrt{d} }) \subset \cdots \subset  K_{n, d}=\mathbb{Q}({\zeta_{2^{n+2}}, \sqrt{d}}) \cdots$$
In this article, we are interested in giving, for $n\geq1$, the length of the $2$-class field tower of $K_{n}$, the structure of the Galois group $\mathrm{Gal}(\mathcal{L}(K_{n})/K_{n})$ and the capitulation of the $2$-ideal classes of $K_{n}$ in its three unramified quadratic extensions for some fields $K$. Our main theorem is the following.

  \begin{theorem}\label{thm the main theorem}
  	Let $d$ be a square-free integer and  $n\geq1$ a positive integer. Let    $m$ be such that $2^m=h_2(-2d)$ and  $K_n=\mathbb{Q}({\zeta_{2^{n+2}}, \sqrt{d}  })$. Let $\mathcal{L}(K_n)$ be the maximal unramified $2$-extension of $K_{n}$ and put $G_n=\mathrm{Gal}(\mathcal{L}(K_{n})/K_{n})$.
  	Then we have
  	\begin{enumerate}[\rm1.]
  		\item If $d=q_1q_2$,  where $q_i\equiv3 \pmod{8}$  are two distinct primes,	then $\mathcal{L}(K_{n})=K_{n}^{(1)}$ and   $G_n$ is abelian isomorphic to    $\mathbb{Z}/2\mathbb{Z}\times\mathbb{Z}/2^{n+m-2}\mathbb{Z}$.
  		\item If $d=p$,  where $p\equiv9\pmod{16}$ is a prime such that   $(\frac{2}{p})_4=1$,  then $G_n$ is abelian isomorphic to $\mathbb{Z}/2\mathbb{Z}\times\mathbb{Z}/2^{n+m-2}\mathbb{Z}$ or modular,  and $\mathcal{L}(K_n)=K_n^{(1)}$ or  $K_n^{(2)}$. In  particular for
  		$n=1$,   	$G_1=Gal(K_{1, d}^{(1)}/K_{1, d})\simeq \mathbb{Z}/2\mathbb{Z}\times\mathbb{Z}/2^{m-1}\mathbb{Z}$ is abelian.
  	\end{enumerate}
  Put $G=\mathrm{Gal}(\mathcal{L}(K_{\infty})/K_{\infty})$. For the   abelian case we have  $G\simeq X_\infty\simeq  \ZZ/2\ZZ \times \ZZ_2$, otherwise $G/G'\simeq X_\infty\simeq  \ZZ/2\ZZ \times \ZZ_2$, where $G'$ is the commutator subgroup of $G$.
  \end{theorem}

  The proof of this theorem is detailed in the next sections. Therein one can find   results on unit groups,  $2$-class groups,  computations of Iwasawa  $\lambda$-invariant and some corollaries.
\section*{Notations}
  Let $k$ be a number field. Throughout this paper, we adopt the following notations.
{\begin{enumerate}[\rm$\star$]
		\item $q_1$,  $q_2$: Two primes congruent to $3\pmod 8$,
		\item $k^+$: The maximal real subfield of  $k$,
		\item $n$: An integer $\geq 1$,
		\item $k_n$: The $n$-th layer of the $\mathbb{Z}_2$-extension of $k$,
		\item $\lambda(k)$: The Iwasawa $\lambda_2$-invariant   of $k$,
		\item   $\lambda(k)^-$: The Iwasawa   $\lambda_2^-$-invariant  of $k$,
		\item $\mathrm{Cl}_2(k)$:  The $2$-class group of $k$,
		\item $rank(\mathrm{Cl}_2(k))$:  The rank of the $2$-class group of $k$,
     	\item $K_{n, d}$:  $\mathbb{Q}({\zeta_{2^{n+2}}, \sqrt{d}  })$,
         \item 	$	F_n$: $\mathbb{Q}( \zeta_{2^{n+2}},  \sqrt{q_1},   \sqrt{q_2})$,
		\item $h(k)$:  The class number of the  field $k$,
        \item $h_2(d)$:  The $2$-class number of the quadratic field $\mathbb{Q}(\sqrt{d})$,
		\item  $\varepsilon_d$: The fundamental unit of the quadratic field $\mathbb{Q}(\sqrt{d})$,
		\item $E_k$: The unit group of $k$,
		\item $q(k):=[E_{k}: \prod_{i}E_{k_i}]$,  with  $k_i$ are  the  quadratic subfields	of $k$,	
		\item $Q_{k}$:   The  Hasse's unit  index, that is  $[E_{k}:W_{k}k^{+}]$, if $k/k^+$ is CM,
        \item FSU: A fundamental system of units,
        \item $N_{K/k}$: The norm map of an extension $K/k$,
        \item $e$: Defined by $[E_{k}:E_{k}\cap N_{K/k}(K^*)]=2^{ e}$, for a $2$-extension $K/k$.	
\end{enumerate}}	

 \section{\textbf{Units and $2$-class numbers  of some multiquadratic  fields}}
  Let us first recall the result stated by Wada \cite{wada} to determine a fundamental system of units of a multiquadratic number field.   Let $K_0$ be a multiquadratic number field. Denote by $\sigma$ and $\tau$ two different generators of the group $\mathrm{Gal}(K_0/\QQ)$, let then $K_1$, $K_2$ and $K_3$ be respectively the invariant subfields of $K_0$ by $\sigma$,  $\tau$ and $\sigma\tau$, and $E_{K_i}$ the unit group of $K_i$. Then the unit group $E_{K_0}$ of $K_0$ is generated by the elements of each $E_i$ and the square roots of elements of the product $E_{K_1}E_{K_2}E_{K_3}$ that are perfect squares in $K_0$.  \label{remainder on units}

  Put  $K=K_0(i)$,  then  to determine  a fundamental system of units of $K$,  we will use the following result (cf. \cite[p. 18]{azizunints99})  that the third author has deduced from a theorem of Hasse \cite[\S 21, Satz 15 ]{Ha-52}.

   \begin{lemma}\label{Lemme azizi} Let $K_0$ be a real number field,  $K=K_0(i)$ a quadratic extension of $K_0$,  $n\geq 2$ be an integer and $\xi_n$ a $2^n$-th primitive root of unity,  then
  	$	\xi_n=\frac{1}{2}(\mu_n+\lambda_ni)$,  where $\mu_n=\sqrt{2+\mu_n}$,  $\lambda_n=\sqrt{2-\mu_{n-1}}$,  $\mu_2=0$,  $\lambda_0=2$ and $\mu_3=\lambda_3=\sqrt{2}$. Let $n_0$ be the greatest
  	integer such that $\xi_{n_0}$ is contained in $K$,  $\{\varepsilon_1,  ...,  \varepsilon_r\}$ a fundamental system of units of $K_0$ and $\varepsilon$ a unit of $K_0$ such that
  	$(2+\mu_{n_0})\varepsilon$ is a square in $K_0\ ($if it exists$)$. Then a fundamental system of units of $K$ is one of the following systems:
  	\begin{enumerate}[\rm 1.]
  		\item $\{\varepsilon_1,  ...,  \varepsilon_{r-1},  \sqrt{\xi_{n_0}\varepsilon } \}$ if $\varepsilon$ exists,   in this case $\varepsilon=\varepsilon_1^{j_1}...\varepsilon_{r-1}^{j_1}\varepsilon_r$,
  		where $j_i\in \{0,  1\}$.
  		\item $\{\varepsilon_1,  ...,  \varepsilon_r \}$ elsewhere.
  		
  	\end{enumerate}
  \end{lemma}
\noindent We will also need the following lemmas about units.
  \begin{lemma}[\cite{Az-00},  {Lemma 5}]\label{1:046}
  	Let $d>1$ be a square-free integer and $\varepsilon_d=x+y\sqrt d$,
  	where $x$,  $y$ are  integers or semi-integers. If $N(\varepsilon_d)=1$,  then $2(x+1)$,  $2(x-1)$,  $2d(x+1)$ and
  	$2d(x-1)$ are not squares in  $\QQ$.
  \end{lemma}

  \begin{lemma}[\cite{ZAT-15}]\label{3:105}
  	Let $d\equiv1\pmod4$ be a positive square free integer and   $\varepsilon_d=x+y\sqrt d$ be the fundamental unit of  $\QQ(\sqrt d)$. Assume   $N(\varepsilon_d)=1$,  then
  	\begin{enumerate}[\rm\indent1.]
  		\item $x+1$ and $x-1$ are not squares in  $\NN$,  i.e.,  $2\varepsilon_{d}$ is not a square in  $\QQ(\sqrt{d})$.
  		\item For all prime  $p$ dividing   $d$,  $p(x+1)$ and $p(x-1)$ are not squares in  $\NN$.
  	\end{enumerate}
  \end{lemma}

  \begin{lemma}\label{lm expressions of units under cond 1}
  Let	$q_1\equiv     	q_2\equiv  3\pmod 8$	 be two primes such that $\left(\dfrac{	q_1}{	q_2}\right)=1$.
  	\begin{enumerate}[\rm 1.]
  		\item Let  $x$ and $y$   be two integers such that
  		$ \varepsilon_{2q_1q_2}=x+y\sqrt{2q_1q_2}$. Then
  		\begin{enumerate}[\rm a.]
  			\item $(x-1)$ is a square in $\NN$,
  			\item 	  $\sqrt{2\varepsilon_{2q_1q_2}}=y_1 +y_2\sqrt{2q_1q_2}$ and 	$2= -y_1^2+ 2q_1q_2y_2^2$,  for some integers $y_1$ and $y_2$ such that $y=y_1y_2$.
  		\end{enumerate}

  		\item  There are two integers     $a$ and $b$   such that
  		$ \varepsilon_{q_1q_2}=a+b\sqrt{q_1q_2}$ and we have
  		\begin{enumerate}[\rm a.]
  			\item $2q_1(a+1)$ is a square in $\NN$,
  			\item  $b$ is even,   $\sqrt{ \varepsilon_{q_1q_2}}=b_1\sqrt{q_1}+b_2\sqrt{q_2}$ and 	$1=q_1b_1^2-q_2b_2^2$  for some integers $b_1$ and $b_2$ such that $b=2b_1b_2$.
  		\end{enumerate}
  		\item  Let    $c $ and $d$ be two integers such that
  		$ \varepsilon_{2q_i}=c +d\sqrt{2q_i}$. Then  we have
  		\begin{enumerate}[\rm a.]
  			\item   $c-1$ is a square in $\NN$,
  			\item  $\sqrt{ 2\varepsilon_{  2q_i}}=d_1 +d_2\sqrt{2q_i}$ and 	$2=-d_1^2+2q_id_2^2$,  for some integers $d_1$ and $d_2$ such that $d=d_1d_2$.
  		\end{enumerate}
  		\item  Let    $\alpha $ and $\beta$ be two integers such that
  		$ \varepsilon_{q_i}=\alpha +\beta\sqrt{q_i}$. Then  we have
  		\begin{enumerate}[\rm a.]
  			\item   $\alpha-1$ is a square in $\NN$,
  			\item  $\sqrt{ 2\varepsilon_{  q_i}}=\beta_1 +\beta_2\sqrt{q_i}$ and 	$2=-\beta_1^2+q_i\beta_2^2$,  for some integers $\beta_1$ and $\beta_2$ such that $\beta=\beta_1\beta_2$.
  		\end{enumerate}
  	\end{enumerate}
  \end{lemma}	
  \begin{proof}
  	\begin{enumerate}[\rm 1.]
  		\item
  		It is known that $N(\varepsilon_{2q_1q_2})=1$, so  $x^2-1=y^22q_1q_2$, thus  by the unique prime factorization theorem in $\mathbb{Z}$ and Lemma \ref{1:046}  we have
  		$$(1):\ \left\{ \begin{array}{ll}
  		x\pm1=2q_1y_1^2\\
  		x\mp1=q_2y_2^2,
  		\end{array}\right.  \quad
  		(2):\ \left\{ \begin{array}{ll}
  		x\pm1=q_1y_1^2\\
  		x\mp1=2q_2y_2^2,
  		\end{array}\right.\quad
  		\text{ or }\quad
  		(3):\ \left\{ \begin{array}{ll}
  		x\pm1=y_1^2\\
  		x\mp1=2q_1q_2y_2^2,
  		\end{array}\right.
  		$$	
  		For some integers $y_1$ and $y_2$   such that $y=y_1y_2$.	
  		\begin{enumerate}[\rm$*$]
  			\item System $(1)$ can not occur. Indeed,
  			\begin{enumerate}[\rm $\bullet$]
  				\item   if $2q_1(x+1)$ is a square,  then  	$-1=\left(\frac{q_2}{q_1}\right)=\left(\frac{x-1}{q_1}\right)=\left(\frac{x+1- 2}{q_1}\right)= \left(\frac{ -2}{q_1}\right) =1.$ Which is absurd.
  				\item  similarly, if $2q_1(x-1)$ is a square,  then 	$-1=\left(\frac{2q_1}{q_2}\right)=\left(\frac{x-1}{q_2}\right)=\left(\frac{x+1- 2}{q_2}\right)= \left(\frac{- 2}{q_2}\right) =1.$ Which is  absurd too.
  			\end{enumerate}
  			\item We similarly eliminate the cases of system $(2)$ and the system  $\left\{ \begin{array}{ll}
  			x+1=y_1^2\\
  			x-1=2q_1q_2y_2^2.
  			\end{array}\right. $	
  		\end{enumerate}	
  		Thus, the only possible case is
  		$\left\{ \begin{array}{ll}
  		x-1=y_1^2\\
  		x+1=2q_1q_2y_2^2.
  		\end{array}\right.$\\	
  		Which implies that  $\sqrt{2\varepsilon_{2q_1q_2}}=y_1 +y_2\sqrt{2q_1q_2}$ and 	$2= -y_1^2+ 2q_1q_2y_2^2$.
  		
  		\item We have $q_1q_2\equiv 1\pmod 8$,  then there are two integers  $\varepsilon_{q_1q_2}=({x+y\sqrt d})/{2}$. 	As  $N(\varepsilon_{q_1q_2})=1$,  so
  		$x^2-4=y^2d$,  hence  $x^2-4\equiv y^2\pmod8$. On the other hand,  if we suppose that $x$ and $y$ are odd,  then $x^2\equiv y^2\equiv1\pmod8$,  but this implies the contradiction $-3\equiv1\pmod8$. Thus $x$ and $y$ are even. It follows that  there are two integers     $a$ and $b$   such that
  		$ \varepsilon_{q_1q_2}=a+b\sqrt{q_1q_2}$.

  		As  $N(\varepsilon_{q_1q_2})=1$, then,  by Lemma \ref{3:105} and  the unique prime factorization theorem, one deduces
  		$$(1):\ \left\{ \begin{array}{ll}
  		a+1=2q_1b_1^2\\
  		a-1=2q_2b_2^2,
  		\end{array}\right.\quad
  		\text{ or }\quad
  		(2):\ \left\{ \begin{array}{ll}
  		a-1=2q_1b_1^2\\
  		a+1=2q_2b_2^2,
  		\end{array}\right.
  		$$	
  		
  		For some integers $b_1$ and $b_2$   such that   $b=2b_1b_2$. As above we show that
  		$ \left\{ \begin{array}{ll}
  		a+1=2q_1b_1^2\\
  		a-1=2q_2b_2^2.
  		\end{array}\right. $ Thus,   $\sqrt{ \varepsilon_{ q_1q_2}}=b_1\sqrt{q_1}+b_2\sqrt{q_2}$ and 	$1=q_1b_1^2-q_2b_2^2$.  		
  	\end{enumerate}
  The items $3$ and $4$ are proved similarly.
  \end{proof}

 %\begin{theorem} \label{lemma on units }
% 	Let     $q_1\equiv q_2\equiv 3\pmod 8$ be two primes. Put $\LL=\mathbb{Q}(\sqrt{2},  \sqrt{q_1}, \sqrt{q_2})$.	Then
% 	$$E_{\LL}=\langle-1, \varepsilon_{  2}, \sqrt{\varepsilon_{   q_1q_2}}  ,    \sqrt{\varepsilon_{   q_1}},  \sqrt{\varepsilon_{   2q_1}},    \sqrt{\varepsilon_{   2q_1q_2}},
 %  \sqrt[4]{ {\varepsilon_{   q_1}}      {\varepsilon_{   q_2}}      {\varepsilon_{   2q_1q_2}} },
 %	\sqrt[4]{ {\varepsilon_{   2q_1}}       {\varepsilon_{   2q_2}}    {\varepsilon_{   2q_1q_2}}} \rangle$$
% 	And the class number of $\LL$ is odd.
% \end{theorem}

Let us now recall the   following class number formula for a multiquadratic number field  which is usually attributed to Kuroda \cite{Ku-50}, but it goes back to Herglotz \cite{He-22}  (cf. \cite[p. 27]{BFL}).
\begin{lemma}\label{wada's f.}
	Let $K$ be a multiquadratic number field of degree $2^n$, $n\in\mathds{N}$,  and $k_i$ the $s=2^n-1$ quadratic subfields of $K$. Then
	$$h(K)=\frac{1}{2^v}[E_K: \prod_{i=1}^{s}E_{k_i}]\prod_{i=1}^{s}h(k_i),$$
	with $$v=\left\{ \begin{array}{cl}
	n(2^{n-1}-1); &\text{ if } K \text{ is real, }\\
	(n-1)(2^{n-2}-1)+2^{n-1}-1 & \text{ if } K \text{ is imaginary.}
	\end{array}\right.$$
\end{lemma}

To use the above lemma we need to recall the values of the class numbers of some quadratic fields.
\begin{lemma}\label{class numbers of quadratic field}
	Let $q=q'\equiv 3\pmod 4$ be two distinct  primes. Then
	\begin{enumerate}[\rm 1.]
		\item By \cite[Corollary 18.4]{connor88}, we have  $h_2(q)=h_2(-q)=h_2(2q)= h_2(2)=h_2(-2)=h_2(-1)=h_2(qq')=1$.
		\item If  $q \equiv 3\pmod 8$, then by \cite[Corollary 19.6]{connor88}, $h_2(-2q)=2$.
		\item If   $q$ or $q'  \equiv 3\pmod 8$, then by \cite[p. 345]{kaplan76},   $h_2(2qq')=2$.
		\item If   $q'  \equiv 3\pmod 8$, with  $\left(\frac{q}{q'} \right)=1$, then by \cite[p. 354]{kaplan76},   $h_2(-qq')=4$. 	
	\end{enumerate}
	
\end{lemma}

\noindent Now we can state the   important result of this section which gives a FSU  of a triquadratic number field.
 \begin{proposition} \label{lemma on units }
 	Let     $q_1\equiv q_2\equiv 3\pmod 8$ be two different prime integers and $\LL=\mathbb{Q}(\sqrt{2},  \sqrt{q_1}, \sqrt{q_2})$.	Then the unit group of $\LL$ is
 %	$$E_{\LL}=\left\langle-1, \varepsilon_{  2} ,    \sqrt{\varepsilon_{   q_1}},  \sqrt{\varepsilon_{   2q_1}} , \sqrt{\varepsilon_{   q_1q_2}} ,    \sqrt{\varepsilon_{   2q_1q_2}},
 %	\sqrt{\sqrt{\varepsilon_{   q_1}}     \sqrt{\varepsilon_{   q_2}}     \sqrt{\varepsilon_{   2q_1q_2}} },
 %	\sqrt{\sqrt{\varepsilon_{   2q_1}}      \sqrt{\varepsilon_{   2q_2}}   \sqrt{\varepsilon_{   2q_1q_2}}} \right\rangle$$
	$$E_{\LL}=\left\langle-1, \varepsilon_{  2} ,    \sqrt{\varepsilon_{   q_1}},  \sqrt{\varepsilon_{   2q_1}} ,     \sqrt{\varepsilon_{   q_2}},
 	\sqrt{\varepsilon_{   q_1q_2}}  ,
 	\sqrt{\sqrt{\varepsilon_{   q_1}}     \sqrt{\varepsilon_{   q_2}}     \sqrt{\varepsilon_{   2q_1q_2}} },
 	\sqrt{\sqrt{\varepsilon_{   2q_1}}      \sqrt{\varepsilon_{   2q_2}}   \sqrt{\varepsilon_{   2q_1q_2}}} \right\rangle,$$
 	and its class number is odd.
 \end{proposition}
 \begin{proof}
 	Without loss of generality we may suppose  that  $\left(\dfrac{	q_1}{	q_2}\right)=1$.
 	Consider the following diagram (Figure \ref{fig:1} below):
 	\begin{figure}[H]
 		$$\xymatrix@R=0.8cm@C=0.3cm{
 			&\LL=\QQ( \sqrt 2,  \sqrt{q_1},  \sqrt{q_2})\ar@{<-}[d] \ar@{<-}[dr] \ar@{<-}[ld] \\
 			L_1=\QQ(\sqrt 2, \sqrt{q_1})\ar@{<-}[dr]& L_2=\QQ(\sqrt 2,  \sqrt{q_2}) \ar@{<-}[d]& L_3=\QQ(\sqrt 2,  \sqrt{q_1q_2})\ar@{<-}[ld]\\
 			&\QQ(\sqrt 2)}$$
 		\caption{Subfields of $\LL/\QQ(\sqrt 2)$}\label{fig:1}
 	\end{figure}

 Put	$ \mathrm{Gal}(\LL/\QQ)=\langle \sigma_1,   \sigma_2,   \sigma_3\rangle$, where
 \begin{center} $\sigma_1(\sqrt{2})=-\sqrt{2}$, \qquad  $\sigma_2(\sqrt{q_1})=-\sqrt{q_1}$, \qquad $\sigma_3(\sqrt{q_2})=-\sqrt{q_2}$,  \\
 	$\sigma_i(\sqrt{2})=\sqrt{2}$ for $i\in\{2,   3\}$,\qquad
 	$\sigma_i(\sqrt{q_1})=\sqrt{q_1}$ for $i\in\{1,   3\}$ and \\
 	$\sigma_i(\sqrt{q_2})=\sqrt{q_2}$ for $i\in\{1,   2\}$.\end{center}
  Hence $L_1$,   $L_2$ and $L_3$ are
 the fixed fields of the subgroups of $\mathrm{Gal}(\LL/\QQ)$ generated respectively by $\sigma_3$,   $\sigma_2$ and $\sigma_2\sigma_3$.
By Lemma \ref{lm expressions of units under cond 1},   we easily infer that
 	 a FSU of $L_1$ is $\{\varepsilon_{  2},  \sqrt{\varepsilon_{   q_1}},  \sqrt{\varepsilon_{   2q_1}}\}$, a
 		 FSU of $L_2$ is $\{\varepsilon_{  2},  \sqrt{\varepsilon_{   q_2}},  \sqrt{\varepsilon_{   2q_2}}\}$ and a
 		 FSU of $L_3$ is $\{\varepsilon_{  2},   \varepsilon_{   q_1q_2} ,  \sqrt{\varepsilon_{   2q_1q_2}}\}$.
 	Thus
 	\begin{eqnarray}\label{eq E1E2E3}
 		E_{L_1}E_{L_2}E_{L_3}=\langle-1, \varepsilon_{  2}, \varepsilon_{   q_1q_2}  ,    \sqrt{\varepsilon_{   q_1}},  \sqrt{\varepsilon_{   2q_1}},   \sqrt{\varepsilon_{   q_2}},  \sqrt{\varepsilon_{   2q_2}},   \sqrt{\varepsilon_{   2q_1q_2}} \rangle.
 	\end{eqnarray}
 	To find a fundamental system of units of $\LL$,  it suffices, as we have said in the beginning of the section, to find elements $\xi$ of $E_{L_1}E_{L_2}E_{L_3}$ which are   squares in $\LL$
 	(cf. page \pageref{remainder on units}). According to Lemma \ref{lm expressions of units under cond 1},  	$\varepsilon_{q_1q_2}$ is a square in $\LL$, then it suffices to look for the  units $\xi$ such that
 	\begin{eqnarray}\label{eq X}
 		\xi^2= \varepsilon_{  2}^a   \sqrt{\varepsilon_{   q_1}}^b \sqrt{\varepsilon_{   2q_1}}^c  \sqrt{\varepsilon_{   q_2}}^d \sqrt{\varepsilon_{   2q_2}}^e  \sqrt{\varepsilon_{   2q_1q_2}}^f.
 	\end{eqnarray}
 	where $a, b, c, d, e$ and $f\in\{0, 1\}$. Let us eliminate  some forms of $\xi^2$ such that $\xi$ can not be in $\LL$. Consider    $L_4=\mathbb{Q}(\sqrt{q_1}, \sqrt{q_2})$,   we will  apply 	the norm     $N_{\LL/L_4}=1+\sigma_1$ to the equation \eqref{eq X}. For this,  by  using Lemma \ref{lm expressions of units under cond 1} we have  the following table:
 	$$\begin{tabular}{ |c|c|c|c|c|c|c|c|}
 	\hline
 	$\varepsilon$&$\varepsilon_{2}$ &$\sqrt{\varepsilon_{q_1}}$&$\sqrt{\varepsilon_{2q_1}}$& $\sqrt{\varepsilon_{q_2}}$&$\sqrt{\varepsilon_{2q_2}}$&$\sqrt{\varepsilon_{2q_1q_2}}$ \\
 	\hline
 	$\varepsilon^{1+\sigma_1}$&$-1$&	${-\varepsilon_{q_1}}$&$1$ &${-\varepsilon_{q_2}}$&$ 1$& $1$ \\
 	\hline
 	\end{tabular}  $$
 	
From which we deduce that
 	\begin{eqnarray*}
 		N_{\LL/L_4}(\xi^2)=N_{\LL/L_4}(\xi)^2&=& (-1)^a\cdot(-1)^b\cdot\varepsilon_{q_1}^b\cdot 1 \cdot (-1)^d \cdot \varepsilon_{q_2}^d \cdot 1\cdot 1, \\
 		&=& (-1)^{a+b+d}\cdot  \varepsilon_{q_1}^b\cdot \varepsilon_{q_2}^d>0.
 	\end{eqnarray*}
 	Thus,  $a+b+d=0\pmod 2$, and since by Lemma \ref{lm expressions of units under cond 1},  $\varepsilon_{q_1}$ and $  \varepsilon_{q_2}$ are not squares in $L_4$ whereas their product is, one deduces that  $b=d$ and $a=0$.
 	Therefore
 	\begin{eqnarray*}
 		\xi^2=    \sqrt{\varepsilon_{   q_1}}^b \sqrt{\varepsilon_{   2q_1}}^c  \sqrt{\varepsilon_{   q_2}}^b \sqrt{\varepsilon_{   2q_2}}^e  \sqrt{\varepsilon_{   2q_1q_2}}^f.
 	\end{eqnarray*}
 	Similarly, we will consider $L_3=\QQ(\sqrt 2,  \sqrt{q_1q_2})$ and apply 	the norm     $N_{\LL/L_3}=1+\sigma_2\sigma_3$. Firstly, we need the following computations:
 	$$\begin{tabular}{ |c|c|c|c|c|c|c|c|c|}
 	\hline
 	 $\varepsilon$&$\varepsilon_2$ &$\sqrt{\varepsilon_{q_1}}$&$\sqrt{\varepsilon_{2q_1}}$& $\sqrt{\varepsilon_{q_2}}$&$\sqrt{\varepsilon_{2q_2}}$&$\sqrt{\varepsilon_{2q_1q_2}}$ \\
 	\hline
 	$\varepsilon^{1+\sigma_2\sigma_3}$  &$\varepsilon_2^2$&$-1$&$-1$&$-1$&$-1$&$\varepsilon_{2q_1q_2}$ \\
 	\hline
 	\end{tabular} $$
 	So	
 	\begin{eqnarray*}
 		N_{\LL/L_3}(\xi^2)&=& (-1)^b\cdot(-1)^c \cdot (-1)^b \cdot (-1)^e  \cdot \varepsilon_{   2q_1q_2}^f, \\
 		&=& (-1)^{b+c+b+e}\cdot  \varepsilon_{   2q_1q_2}^f>0.
 	\end{eqnarray*}
 	Thus all what we can deduce is that $c=e$.  Hence
 	\begin{eqnarray*}
 		\xi^2=    \sqrt{\varepsilon_{   q_1}}^b \sqrt{\varepsilon_{   2q_1}}^c  \sqrt{\varepsilon_{   q_2}}^b \sqrt{\varepsilon_{   2q_2}}^c  \sqrt{\varepsilon_{   2q_1q_2}}^f.
 	\end{eqnarray*}
 	Similarly, we consider $L_5=\mathbb{Q}(\sqrt{q_1}, \sqrt{2q_2})$ and we apply $N_{\LL/L_5}=1+\sigma_1\sigma_3$. By Lemma \ref{lm expressions of units under cond 1} we have
 	$$\begin{tabular}{ |c|c|c|c|c|c|c|c|c|}
 	\hline
 	 $\varepsilon$& $\varepsilon_2$&$\sqrt{\varepsilon_{q_1}}$&$\sqrt{\varepsilon_{2q_1}}$& $\sqrt{\varepsilon_{q_2}}$&$\sqrt{\varepsilon_{2q_2}}$&$\sqrt{\varepsilon_{2q_1q_2}}$ \\
 	\hline
 	$1+\sigma_1\sigma_3$ &$-1$ &$-\varepsilon_{q_1}$&$1$&$1$&$-\varepsilon_{2q_2}$&$-\varepsilon_{2q_1q_2}$ \\
 	\hline
 	\end{tabular}$$
 	Then
 	\begin{eqnarray*}
 		N_{\LL/L_5}(\xi^2)&=&   (-1)^b\cdot\varepsilon_{q_1}^b\cdot 1 \cdot 1 \cdot (-1)^c\cdot \varepsilon_{2q_2}^c   \cdot(-1)^f\cdot \varepsilon_{   2q_1q_2}^f, \\
 		&=& (-1)^{ b+c+f}\cdot\varepsilon_{q_1}^b\cdot \varepsilon_{2q_2}^c\cdot \varepsilon_{   2q_1q_2}^f>0.
 	\end{eqnarray*}
 	Thus all what we can deduce is 	that $ b+c+f\equiv 0\pmod 2$. Applying the norm maps on all other subextensions of $\LL$ we deduce no more information  on  $b, c, e$ and $f$.
 	On the other hand,   by  Lemmas \ref{wada's f.} and \ref{class numbers of quadratic field} we have
 	\begin{eqnarray}
 		h_2(\LL)&=&\frac{1}{2^{9}}q(\LL)  h_2(2) h_2(q_1) h_2(2q_1)h_2(q_2) h_2(2q_2)h(q_1q_2)  h_2(2q_1q_2)\nonumber \\
 		&=&\frac{1}{2^{9}}\cdot q(\LL)\cdot 1\cdot 1 \cdot 1  \cdot 1 \cdot 1 \cdot 1 \cdot 2= \frac{1}{2^{8}}\cdot q(\LL).\label{frm wadas formula}
 	\end{eqnarray}
 	If we suppose that   $b=c=f=0$, then there is no element of $E_{L_1}E_{L_2}E_{L_3}$ which is a square in $\LL$ except ${\varepsilon_{q_1q_2}}$. So
 	$q(\LL)=2^6$.  Hence formula  \ref{frm wadas formula}, implies that $h_2(\LL)=\frac{1}{2^{8}}\cdot 2^6=\frac{1}{2^{4}}$.
 	Which is absurd. Thus the case  $b=c=f=0$ is impossible.  Hence $\xi^2$ can take one of the following forms:
 	\begin{eqnarray} \label{ from 1 of xi}
 		\xi_1^2=    \sqrt{\varepsilon_{   q_1}}     \sqrt{\varepsilon_{   q_2}}     \sqrt{\varepsilon_{   2q_1q_2}}\quad  (\text{i.e.,  } c=0 \text{ and } b=f=1),
 	\end{eqnarray}
 	
 	\begin{eqnarray} \label{ from 2 of xi}
 		\xi_2^2=      \sqrt{\varepsilon_{   2q_1}}      \sqrt{\varepsilon_{   2q_2}}   \sqrt{\varepsilon_{   2q_1q_2}}\quad   (\text{i.e.,  } b=0 \text{ and } c=f=1),
 	\end{eqnarray}
 	
 	\begin{eqnarray} \label{ from 3 of xi}
 		\xi_3^2=    \sqrt{\varepsilon_{   q_1}}  \sqrt{\varepsilon_{   2q_1}}   \sqrt{\varepsilon_{   q_2}}  \sqrt{\varepsilon_{   2q_2}}\quad    (\text{i.e.,  } f=0 \text{ and } c=b=1).
 	\end{eqnarray}

 	Remark that by the equalities \eqref{eq E1E2E3} and \eqref{frm wadas formula} at least two   equations from \eqref{ from 1 of xi},  \eqref{ from 2 of xi} and \eqref{ from 3 of xi} are verified for some $\xi \in \LL$ (elsewhere,  we get a contradiction as for the case $b=c=f=0$).
 	Remark also that if two of  the forms in the equations \eqref{ from 1 of xi},  \eqref{ from 2 of xi} and \eqref{ from 3 of xi},  are verified for some $\xi_i$ and $\xi_j\in \LL$ ($i\not= j$),
 	then by their product we may deduce that the third form is also verified for some $\xi_k\in \LL$. So the result.	
 \end{proof}

\noindent We close this section by the following  result on class numbers of some number fields.
   \begin{lemma}\label{lm class number of the 1st step is odd}
 	Let     $q_1\equiv q_2\equiv 3\pmod 8$ be two distinct primes. Then
 	\begin{enumerate}[\rm 1.]
 		\item The class numbers of $F= \mathbb{Q}(\sqrt{q_1},   \sqrt{q_2},   i)$ and
 		$ F^+=\mathbb{Q}(\sqrt{q_1},   \sqrt{q_2})$ are odd.
 		\item The $2$-class number of $K=\mathbb{Q}(\sqrt{-q_1}, \sqrt{q_2}, \sqrt{2})$ equals  $h_2(-2q_1q_2)$.
 	\end{enumerate}
 	 	\end{lemma}
 \begin{proof}
 \begin{enumerate}[\rm 1.]
 	\item   	Note first that  by Lemma \ref{lm expressions of units under cond 1}, we have $ \sqrt{2\varepsilon_{q_1}}$,    $\sqrt{2\varepsilon_{q_2}}\in F^+$,    so  $ \sqrt{ \varepsilon_{q_1} \varepsilon_{q_2}}\in F^+$ and $ \sqrt{   \varepsilon_{q_i}}\not\in F^+$. Thus    a fundamental system of unit  of $F^+$ is
 	$\{\varepsilon_{q_1},    \sqrt{\varepsilon_{q_1}\varepsilon_{q_2}},    \sqrt{\varepsilon_{q_1q_2}}\}$  and  that of $F$ is $\{   \sqrt{\varepsilon_{q_1}\varepsilon_{q_2}},    \sqrt{\varepsilon_{q_1q_2}},  \sqrt{i\varepsilon_{q_1}}\}$ (see    Lemma \ref{Lemme azizi}). Therefore,
 	 $q(F)=2^3$. Thus     by Lemmas  \ref{wada's f.} and \ref{class numbers of quadratic field}  we obtain
 	\begin{eqnarray*}
 		h_2(F)&=&\frac{1}{2^5}q(F)h_2(q_1)h_2(q_2)h_2(-q_1)h_2(-q_2)h_2(q_1q_2)h_2(-q_1q_2)h_2(-1),   \\
 		&=&\frac{1}{2^5}q(F)\cdot 4 =1.
 	\end{eqnarray*}
   We  similarly prove  that the $2$-class number of $F^+$, is odd. Hence the first assertion of the lemma.
 	\item For $K=\mathbb{Q}(\sqrt{-q_1}, \sqrt{q_2}, \sqrt{2})$, by Lemma \ref{lm expressions of units under cond 1}, we have $\{\varepsilon_{  2}, \sqrt{\varepsilon_{  q_2}},  \sqrt{\varepsilon_{  2q_2}}\}$ is a FSU of
 	$K^+=\mathbb{Q}(  \sqrt{q_2},\sqrt{2})$. One can check  easily that $q_1\varepsilon$ is not a square in  $K^+$, for all $ \varepsilon\in   \{\varepsilon_{  2}^i \sqrt{\varepsilon_{  q_1}}^j  \sqrt{\varepsilon_{  2q_2}}^k, \text{ with } i, j\text{ and } k\in\{0,1\}\}$.  It follows by   \cite[Proposition  3]{azizunints99}, that $Q_K=1$. Thus $q(K)=4$.
 	  So using  Lemmas  \ref{wada's f.} and \ref{class numbers of quadratic field} we get
	
 	$\begin{array}{ll}
 	h_2(K)&=\frac{1}{2^{5}}q(K)   h_2(-q_1) h_2(q_2)  h_2(-2q_1) h_2(2q_2)    h_2(-q_1q_2)   h_2(-2q_1q_2) h_2(2)   \\
 	&=\;\frac{1}{2^{5}}\cdot 4\cdot   1 \cdot 1 \cdot 2 \cdot 1 \cdot 4 \cdot h_2(-2q_1q_2)\cdot 1, \\
 	&= h_2(-2q_1q_2).
 	\end{array}$
 \end{enumerate}	
Which completes the proof.
 \end{proof}

 \section{\textbf{Some Iwasawa theory results}}
Let $K$ be a number field and $p$ a prime number. 	Let 	$K_\infty$ be a   $\mathbb Z_p$-extension of $K$.\label{ zp extension}
 Denote by  $K_n$, with $n\geq0$ an integer, the  unique cyclic subfield of $K_\infty$ of degree $p^n$ over $K$, i.e. the $n$-th	layer of the $\mathbb Z_p$-extension of $K$.
In \cite{iwasawa59},    Iwasawa proved that for $n$  sufficiently large,  the highest power of $p$ dividing the class number of $K_n$ is $p^{\lambda_p n+\mu_p p^n+\nu_p}$
for some integers $\lambda_p$,   $\mu_p\geq 0$ and  $\nu_p$,  all independent of $n$,    called the Iwasawa invariants of $K_\infty$. Later,  for abelian number fields,   Ferrero and    Washington proved that $\mu_p=0$ (cf. \cite{ferrerowashington}). Denote by  $A_p(K_n)$ (or   simply $A_n$) the $p$-Sylow subgroup of the ideal class group of $K_n$, and by  $L(K_\infty)$ (resp. $L(K_n)$)  the Hilbert $p$-class field of $K_\infty$ (resp. $K_n$), then $X=\mathrm{Gal}(L(K_\infty)/K_\infty)=\displaystyle\lim_{\longleftarrow}A_p(K_n)$ is  a module over the formal power series ring $\Lambda=\mathbb{Z}_p[[T]]$ (here the action of $1 + T$ is defined by  an inner automorphism induced
from a fixed topological generator of $\mathrm{G}(K_\infty/K)$). For  $m \geq n \geq0$, put $\omega_n =(1+T)^{p^n}-1\in \Lambda$, and $\nu_{m, n}=\omega_m/\omega_n\in \Lambda$.

 For $p=2$,  the tower
$K\subset K_1=K(\sqrt{2} )\subset K_2=K(\sqrt{2+\sqrt{2}})\subset
\cdots$
 defines a $\mathbb Z_2$-extension called the cyclotomic $\mathbb Z_2$-extension of $K$ of layers $K_i$. We need the following lemmas.
 \begin{lemma}[\cite{chemsZkhnin2}]\label{lm deco to 2}
 	Let $n\geq1$ and    $q$  a  prime such that $q\equiv3  \pmod 8$. Then $q$ decomposes into the product of  $2$ prime ideals of $K_n=\mathbb Q{(\zeta_{2^{n+2}})}$ while it is inert in $K_n^+$.
 \end{lemma}
  \begin{lemma}[\cite{chemsZkhnin2}]\label{lm parity of class number}
 	Let  $q$  be  prime such that $q\equiv3  \pmod 8$. Then  the class number of $\mathbb Q{(\zeta_{2^{n+2}},  \sqrt{q})}$ is odd.
 \end{lemma}
\noindent Let $rank_2(G)$ denote the $2$-rank of a group $G$.
 \begin{lemma}[\cite{fukuda}]\label{lm fukuda}
	Let $K/k$ be a $\mathbb{Z}_2$-extension,  $k_n$ its    n-th layer  and $n_0$  an integer such that any prime of $K$ which is ramified in $K/k$ is totally ramified in $K/k_{n_0}$.
	%	\begin{enumerate}[\rm 1.]
	%\item If there exists an integer $n\geq n_0$ such that such that $h_2(k_n)=h_2(k_{n+1})$,  then $h_2(k_n)=h_2(k_{m})$ for all $m\geq n$.
	%\item
	If there exists an integer $n\geq n_0$ such that $rank_2(\mathrm{Cl}(k_n))= rank_2( \mathrm{Cl}(k_{n+1}))$,  then
	$rank_2(\mathrm{Cl}(k_{m}))= rank_2(\mathrm{Cl}(k_{n}))$ for all $m\geq n$.
	%\end{enumerate}
\end{lemma}
 % Let us recall following formula of Kida.
 \begin{lemma}\cite{kida}\label{Kida}
 	Let $K$ and $F$ be   $CM$-fields and $F/K$ a finite $2$-extension. Assume that $\mu^-(K)=0$, then $\mu^-(F)=0$ and
 	\[\lambda^-(F)-\delta(F)=[F_{\infty}:K_{\infty}]\left(\lambda^-(K)-\delta(K)\right)+\sum (e_{\beta}-1)-\sum(e_{\beta^+}-1), \]
 	where $\delta(k)=1$ or $0$ according to whether  $K_\infty$ contains the fourth roots of unity or not,  $e_{\beta}$ is the ramification index of a prime $\beta$ in $F_{\infty}/K_{\infty}$ coprime to $2$ and $e_{\beta^+}$ is the ramification index of a prime coprime to $2$ in $F_{\infty}^+/K_{\infty}^+$.
 \end{lemma}

\noindent Let us close this section with the following important lemma.
% \begin{proposition}\label{lm  proved by of katharina}
%	Let $K/\mathbb{Q}$ be an abelian  extension such that $h(K)$ is odd,   the
%	$p$-class group of $K_1$ is cyclic and $\lambda_p(K)=1$. If   every prime  ramified in $K_\infty/K$ is totally ramified,  then the $p$-class group of $K_n$ is cyclic for all $n\geq 1$.
%\end{proposition}
%\begin{proof}
% For $A_n$,  the $2$-class group of $K_n$, we know that  $\displaystyle X=\lim_{\longleftarrow}A_n$, then there exists  a submodule $Y\subset X$ such that $A_n=X/\nu_{n, 0}Y$. As $A_0$ is trivial we see that $X=Y$. Let $E$ be the elementary $\Lambda$-module associated to $X$. As $E$ is isomorphic to $\ZZ_p$ we can assume that $X\to E$ is surjective. The kernel is by definition the maximal finite submodule $F$ of $X$.
%	As $A_1$ is cyclic and maps surjectively to $E/\nu_{1, 0}E$ we see that $F\subset \nu_{1, 0}X$. Let $a=\nu_{1, 0}b\in F$ for some $b\notin F$.
%	Let $e$ be the image of $b$ in $E$. Then $\nu_{1, 0}e=0$. Hence $\nu_{1, 0}$ lies in the annihilator of $E$. As $K_{\infty}/K$ is totally ramified the annihilator of $E$ is coprime to all $\nu_{n, 0}$ and in particular to $\nu_{1, 0}$ (cf. \cite[p. 283]{washington1997introduction}). It follows that there is no finite submodule and $X$ is in fact isomorphic to $E$. In particular,  it is cyclic as $\ZZ_p$-module and hence $A_n$ is cyclic for all $n$.	
%\end{proof}

  \begin{lemma}\label{lm C(F_n) is cycli}
 	Let $q_1\equiv q_2 \equiv 3\pmod  8$ be two primes and $n\geq 1$. Let $F_n=\mathbb{Q}( \zeta_{2^{n+2}},  \sqrt{q_1},   \sqrt{q_2})$. Then we have
 	%Then,  for all $n\geq 0$,  the $2$-class group of $F_n$ is cyclic non trivial.
 	\begin{enumerate}[\rm 1.] 	
 		\item   $\lambda(F)=1$,  where $F= \mathbb{Q}(\sqrt{q_1},   \sqrt{q_2},   i)$.
 		\item The class numbers of the layers of the $\mathbb{Z}_2$-extension of $F^+= \mathbb{Q}(\sqrt{q_1},   \sqrt{q_2})$ are odd.
 		\item For all $n\geq 1$,  the $2$-class group of $F_n$ is cyclic non trivial.
 	\end{enumerate}
 \end{lemma}

 \begin{proof}
 	Without loss of generality we may assume that $ q_1$ and $ q_2$ satisfy the   condition $\left(\dfrac{	q_1}{	q_2}\right)=1$.
 	Set $k=\mathbb{Q}(\sqrt{q_1}, i)$. Let us firstly prove that the $2$-class group of $F_1$ is cyclic non trivial. It is easy to see that there      are exactly two primes of $F$ that ramify in  $F_1$.
 	Since, by Lemma \ref{lm class number of the 1st step is odd}, the class number of $F$ is odd, so
 	by  ambiguous class number formula (cf. \cite{Gr}) $rank(\mathrm{Cl}_2(F_1))=2-1-e=1-e$,  where
 	$e$ is defined by $[E_{F}: E_{F}\cap N_{F_1/F}(F_1)]=2^e$. Then the $2$-class group of $F_1$ is trivial or  cyclic non-trivial.
 	
 	Since $\left(\dfrac{q_1}{q_2}\right)=1$,  then $q_2$ decomposes into the product of $2$ primes of $\mathbb{Q}(\sqrt{q_1})$. So by Lemma \ref{lm deco to 2},  $q_2$ decomposes
 	into the product of $4$ primes in $k_n(\sqrt{q_1})$ while  it decomposes into $2$ primes in  $k_n^+(\sqrt{q_1})$. Since
 	$[F_{\infty}:k_{\infty}]=2$,  then by Lemma  \ref{Kida},  we have
 	\begin{center}$\lambda^-(F)-1=2\cdot\left(0-1\right)+4-2,$ and thus  $\lambda^-(F)=1$.\end{center}
 	
 	Assume that the $2$-class group of $F_1$ is trivial. Note that the layers of the cyclotomic  $\mathbb{Z}_2$-extension of $F$ are given by
 	$	F_n=\mathbb{Q}( \zeta_{2^{n+2}},  \sqrt{q_1},   \sqrt{q_2})$.
 	Then by Lemma \ref{lm fukuda},    the $2$-class group of $F_n$ is trivial for all $n\geq1$. Thus $\lambda(F)=0$,  which contradicts the fact that   $\lambda^-(F)=1$. Thus the $2$-class group of $F_1$
 	  is cyclic non-trivial. By Lemma   \ref{lm class number of the 1st step is odd}  and  Proposition \ref{lemma on units },  the class numbers of $F^+$ and $F_1^+$ are odd,
 	   so by Lemma \ref{lm fukuda} the class number  of $	F_n^+$ is odd for all $n\geq 1$. Therefore  $\lambda(F)^+=0$ and  $\lambda(F)=\lambda(F)^-=1$.  We
 	   adopt the definitions of $A_n^+$ and   $A_n^-$ given in \cite{Katharina}. Note that  there is no finite submodule in $A_\infty^-$ (cf. \cite[Theorem 2.5]{Katharina}).
 	   
 	  Since $h(F_n^+)$ is odd for all $n$, then by the definition of $A_n^+$ is trivial or contains only the primes above $2$ that are ramified in $F_n/F_n^+$. Note that index of ramification of 
 	  $2$ in $F_1$ (resp. $F_1^+$) equals  $4$. Therefore the primes above $2$ are unramified in $F_1/F_1^+$ and so they are also unramified in $F_n/F_n^+$, for all $n$. Then  
 	   $A_n^+$  is trivial. Hence there  is no  finite part in $A_\infty $. Thus  $A_n $ is cyclic for all $n\geq 1$.

 \end{proof}
 \section{\textbf{Proof of the main Theorem}}
 Let us quote the last ingredient  needed to prove  the main theorem.

\begin{lemma}[\text{\cite{d.hilbet}}]\label{lm hilbert}
	Let $K/k$ be a quadratic extension and $\mu\in k$ prime to $2$ such that $K=k(\sqrt{\mu})$. The extension $K/k$ is unramified at  finite primes if and only if $\mu$ verify the following properties:
	\begin{enumerate}[\rm 1.]
		\item The principal ideal generated by $\mu$ is a square of a fractional ideal of $k$.
		\item There exists $ \xi\in k$ such that $\mu\equiv  \xi^2\pmod 4$.
	\end{enumerate}
\end{lemma}

 \begin{proof}[\textbf{Proof of Theorem \ref{thm the main theorem} }]
 		We know, by \cite{chemskatharina},  that the $2$-class group of $K_{n}=K_{n, d}$ is isomorphic to
 		$ \mathbb{Z}/2\mathbb{Z}\times\mathbb{Z}/2^{n+m-2}\mathbb{Z} $, whenever $d$ takes one of the forms appearing in Theorem \ref{thm the main theorem}. Let us start by proving the assertions of the first item.
 	\begin{enumerate}[\rm 1.]	\item   Note that $F_n$ is the genus field of $K_{n, d}$,  $[F_n:K_{n, d}]=2$ and its  $2$-class group  is cyclic,  then by class field theory the first Hilbert $2$-class field $F_n^{(1)}$ of $F_n$ coincides with  the second Hilbert $2$-class field $K_{n, d}^{(2)}$ of $K_{n, d}$.
 	
 	 To this end, we claim that  $ K_{n, d}^{(1)}=F_n^{(1)}=K_{n, d}^{(2)}$. To prove this  it suffices then to verify  that $h_2(F_n)=\frac{h_2(K_{n,d})}{2}=2^{n+m-2}.$
 	
 	   Note that by Lemma \ref{lm C(F_n) is cycli},  $A_n(F_n)$, the $2$-class group of $F_n$ is cyclic and that $\lambda=1$.
 	  Therefore, the $\Lambda$-module $A_{\infty}=\displaystyle\lim_{\longleftarrow}(A_n(F_n))$  is   isomorphic to its elementary module. Since every prime which ramify in $F_\infty/F$ is totally ramified, then according to \cite[pp. 282-283]{washington1997introduction}, one gets  $\nu_{n,0}A_{\infty}=2\nu_{n-1,0}A_{\infty}$ for all $n\ge1$.
 	  Hence, $$|A_n|=\vert A_{\infty}/\nu_{n,0}A_{\infty}\vert=\vert A_{\infty}/2^{n-1}A_{\infty}\vert \vert A_{\infty}/\nu_{1,0}A_{\infty}\vert=2^{n+s},$$
 	 for $n\geq 1$ and some constant $s\ge 0$. Hence the Iwasawa formula holds for all  $n\geq 1$, from which we infer that  $h_2(F_1)=2^{1+s}$.
 	 It follows,
 	  by Lemmas  \ref{wada's f.} and \ref{class numbers of quadratic field} that \\
 	 $\begin{array}{ll}
 	 h_2(F_1)&=\frac{1}{2^{16}}q(F_1) h_2(q_1) h_2(-q_1) h_2(q_2) h_2(-q_2) h_2(2q_1)h_2(-2q_1) h_2(2q_2)\\
 	 &\qquad h_2(-2q_2) h_2(q_1q_2) h_2(-q_1q_2) h_2(2q_1q_2) h_2(-2q_1q_2) h_2(2) h_2(-2)h_2(-1) \\
 	 &=\;\frac{1}{2^{16}}\cdot q(F_1)\cdot   1 \cdot 1 \cdot 1 \cdot 1 \cdot 1\cdot  2 \cdot1\cdot 2\cdot 1\cdot 4 \cdot2\cdot h_2(-2q_1q_2)\cdot 1 \cdot 1 \cdot 1, \\
 	 &= \frac{1}{2^{11}}q(F_1)h_2(-2q_1q_2).
 	 \end{array}$\\
 	We have to compute $q(F_1)$, i.e., $Q_{F_1}$. Note that $Q_{F_1}$ must be $2$, for if  $Q_{F_1}=1$, then by Proposition \ref{lemma on units }, $q(F_1)=2^9$. So $h_2(F_1)=\frac{1}{4}h_2(-2q_1q_2)$.
 	  Considering the field  $K=\mathbb{Q}(\sqrt{-q_1}, \sqrt{q_2}, \sqrt{2})$,
 	 since $F_1/K$ is an unramified quadratic extension, then $h_2(F_1)\geq \frac{1}{2}h_2(K) $, and thus  by Lemma \ref{lm class number of the 1st step is odd}, we get
 	 $h_2(F_1)\geq  \frac{1}{2}h_2(-2q_1q_2)$. So we have a contradiction.  Hence $Q_{F_1}=2$,  and  so by Lemma  \ref{Lemme azizi} and Proposition \ref{lemma on units },  $q(F_1)=2^{10}$.
 	 	Therefore, $h_2(F_1)=\frac{h_2(K_{1,d})}{2}=2^{1+m-2}=2^{1+s}$, this in turn implies, using Iwasawa formula, that  $h_2(F_n)=2^{n+m-2}=\frac{h_2(K_{n,d})}{2}$. 
 	 	Hence  $ K_{n, d}^{(1)} =K_{n, d}^{(2)}$. Which gives the first item of the main theorem.
 \item Now assume that we are in the conditions of the second item of the main theorem.
  Put $K_n'=\mathbb Q{(\zeta_{2^{n+2}})}$,  $p=a^2+16b^2=e^2-32f^2$ and $\pi_1'=a+4bi$, $\pi_1''=a-4bi$,   $\pi_2'= e+4f\sqrt{2}$ and $\pi_2''= e-4f\sqrt{2}$.

  So $p=\pi_r'\pi_r''$, for $r\in\{1, 2\}$.   Note   that by \cite[Proposition 1]{chems}, $p$ splits into $4$ ideals of $K_n'$.  Thus, there are exactly  $2$ primes of
  $K_n'$ above $\pi_r'$ (resp. $\pi_r''$).
  Note also that by \cite[Lemma 6]{chems},  we have the following   norm residue symbols
  $\left( \frac{\zeta_{2^{n+2}},p}{\mathfrak p_{K_{n}'}}\right)=-1=\left( \frac{\zeta_{2^{n+2}},\pi_r'\pi_r''}{\mathfrak p_{K_n'}}\right)$. So choose $\pi_r\in\{\pi_r',\pi_r''\}$ such that
  $\left( \frac{\zeta_{2^{n+2}},\pi_r}{\mathfrak p_{K_n'}}\right)=-1 $, where $\mathfrak p_{K_n'}$ is a prime ideal of $K_n'$ above $\pi_r$.

   Since the primes  of $K_n'$ that ramify    in $K_{n,p}$ are exactly the divisors  of  $p$ in $K_n'$, then the ideals of $K_{n,p}$ generated by $\pi_1$ and $\pi_2$ are squares of ideals of $K_{n,p}$. As $a\equiv e \equiv \pm 1\equiv i^2\pmod 4$, since they are odd. It follows that the equations $\pi_r\equiv \xi^2\pmod 4$, $r=1$ or $2$, have   solutions.
   Therefore, By Lemma  \ref{lm hilbert},	$L_{n,1}=K_{n,p}(\sqrt{\pi_1})$ and  $L_{n,2}=K_{n,p}(\sqrt{\pi_2})$  are two distinct unramified  quadratic extensions of $K_{n,p}$.

   Let us now show that the $2$-class groups of $K_1$ and $K_2$ are cyclic. Put   $k_{n,r}=\mathbb{Q}(\zeta_{2^{n+2}},\sqrt{\pi_r })$.
    By the   ambiguous class  number formula (cf. \text{\cite{Gr}}) applied on the extension $k_{n,r}/K_n'$, we have
   $rank(\mathrm{Cl}_2(k_{n,r}))=2-1-e=1-e$. Since $\left( \frac{\zeta_{2^{n+2}},\pi_r}{\mathfrak p_{K_n'}}\right)=-1 $, then $e\not=0$. Hence
   $rank(\mathrm{Cl}_2(k_{n,r}))=0$. So the class number of $k_{n,r}$ is odd. Then again by the   ambiguous class  number formula (cf. \text{\cite{Gr}}) applied on the extension $L_{n,r}/k_{n,r}$,
    we have $rank(\mathrm{Cl}_2(L_{n,i}))=2-1-e=1-e$. Thus $\mathrm{Cl}_2(L_{n,r})$, is either trivial or cyclic non trivial, but since $\mathrm{Cl}_2(K_{n,p})$ is of type $(2,2^{2^{n+m-2}})$ and $L_{n,r}/K_{n,p}$ is an unramified quadratic extension,
    we infer that
    $\mathrm{Cl}_2(L_{n,r})$, can not be trivial. Hence $\mathrm{Cl}_2(L_{n,r})$ is cyclic. It follows that the second $2$-class group of $K_{n,p}$ is abelian or modular (cf. \cite[Theorem 12.5.1]{HallM}).

 Assume now that $n=1$, we use the fact that   $K_{1, p}$ is the genus field of $\mathbb{Q}(\sqrt{2p}, i)$. As, by \cite[Théorème 5]{taous2008},  the second $2$-class group of $\mathbb{Q}(\sqrt{2p}, i)$ is abelian, then by class field theory
 	the second $2$-class group of $K_{1, p}$ is abelian too. Which  completes the proof of the main Theorem.
 	\end{enumerate} \end{proof}

 \begin{remark}
Keep the notations of the second part of the above proof. Then 	$L_{n,r}$, with $r=1$ or $2$, are two unramified quadratic extensions of 	$K_{n,p}$, for which the $2$-class groups are
cyclic.
 \end{remark}
 \begin{remark}Let  $d$ be in one of the forms appearing in Theorem \ref{thm the main theorem}.
 Since by \cite[Proposition 2]{mccall1995imaginary}, the $2$-class group of $K_{0,d}:=K=\mathbb{Q}(\sqrt{d}, i)$ is cyclic, then the $2$-class field tower of $K_{0,d}$ terminates at the first step.
\end{remark}

\noindent Now we treat the capitulation of the $2$-classes of each $K_n$ and the orders of the three unramified
 	quadratic extensions of $K_n$. We need  the following lemma:
 \begin{lemma}[\cite{benlemsne}]\label{lm benjLS}
 	Let $k$ be a number field such that $\mathrm{Cl}_2(k)\simeq (2^m, 2^n)$ for some positive integers $m$ and $n$. If there is an unramified
 	quadratic extension of $k$  with $2$-class number $2^{m+n-1}$,  then all three unramified quadratic extensions of $k$ have $2$-class number
 	$2^{m+n-1}$,  and the $2$-class field tower of $k$ terminates at $k^{(1)}$.
 	
 	Conversely, if the $2$-class field tower of $k$ terminates at $k^{(1)}$, then all three
 	unramified quadratic extensions of $k$ have $2$-class number $2^{m+n-1}$.
 \end{lemma}

\noindent By the main Theorem,   Lemma \ref{lm benjLS} and \cite{bensney} (or \cite{benlem06}), one easily deduces the following corollary.
\begin{corollary}\label{cor first cor}
	Let $d$ be a positive  square-free integer and $m$ such that $h_2(-2d)=2^m$.
		\begin{enumerate}[\rm 1.]
		\item If $d=q_1q_1$,  where $q_i\equiv3 \pmod{8}$ are two distinct primes,  	then for all $n\geq 1$
		the class numbers of the three unramified quadratic extensions of $K_{n, d}$ equal  $h_2(F_n)=2^{n+m-2} $.
		Furthermore,   there are exactly $4$ classes of  the $2$-class group  of $K_{n, d}$ capitulating  in each unramified quadratic  extension of $K_{n, d}$.
\item Assume $d=p\equiv9\pmod{16}$ is a prime such that   $(\frac{2}{p})_4=1$.
 \begin{enumerate}[\rm a.] \item In the abelian case, for all $n\geq 1$,
		the class numbers of the three unramified quadratic extensions of $K_{n, d}$ equal  $2^{n+m-2} $, and
		 there are exactly $4$ classes of  the $2$-class group  of $K_{n, d}$ capitulating  in each unramified quadratic  extension of 	$K_{n, d}$.
		 \item In the modular case and  	for $n\geq 2$,
		    there are only $2$  classes of  the $2$-class group of $K_{n, d}$  capitulating in each  unramified quadratic extensions of $K_{n, d}$.
		\end{enumerate}
	\end{enumerate}
\end{corollary}

\begin{corollary}Let $q_1\equiv q_2 \equiv 3\pmod  8$ be two primes. Then we have
	$$\mathrm{Cl}_2(F_n)\simeq  \mathbb{Z}/2^{n+m-2}\mathbb{Z},  $$
	for all $n\geq 1$.
\end{corollary}

\begin{remark}
	Let $d$ be a square-free integer such that the $2$-class group of $K_{1,d}=\mathbb{Q}(\sqrt{d}, \sqrt{2}, \sqrt{-1})$ is of type
	$(2, 4)$, for the conditions that make the $2$-class group of $K_{1, d}$ of type $(2, 4)$ see \cite{chemsZkhnin2class}. Then by Theorem \ref{thm the main theorem}  the class field tower of $K_{1, d}$ stops at the first step $K_{1, d}^{(1)}$ and   $G=\mathrm{Gal}(K_{1, d}^{(1)}/K_{1,d})$ is abelian.	Therefore, by \cite[page 112]{ATZ-16}, $4$ classes (resp. all the classes) of  the $2$-class group  of $K_{n, d}$ capitulate  in each unramified quadratic (resp. biquadratic) extension of 	$K_{n, d}$.
\end{remark}

   \section*{\textbf{Acknowledgments}}
  The   authors would like to express their gratitude to  Katharina M\"uller   for her  comments on the previous versions of this paper. 
  Particularly the first author would like to thank her for explaining to him many things in Iwasawa Theory.

\end{document}